\documentclass[12pt]{article}
\usepackage{amsfonts,latexsym,rawfonts,amsmath,amssymb,amsthm}
\usepackage{amsmath,amssymb,amsfonts,latexsym,lscape,rawfonts}
\usepackage[cm]{fullpage}
\usepackage{amsmath,amscd, float,times,rotating}
\usepackage{pb-diagram}
\usepackage{color}
\usepackage{hyperref}
\numberwithin{equation}{section}

\newtheorem{thm}{Theorem}[section]
\newtheorem{cor}[thm]{Corollary}
\newtheorem{lem}[thm]{Lemma}

\newtheorem{rem}[thm]{Remark}

\newcommand{\norm}[1]{\left\Vert#1\right\Vert}
\newcommand{\abs}[1]{\left\vert#1\right\vert}

\newcommand{\Real}{\mathbb R}

\newcommand{\pfrac}[2]{\frac{\partial #1}{\partial #2}}

\title{{ On stability of the hyperbolic space form under the normalized Ricci flow
}}
\author{Haozhao Li$^{*}$ and  Hao Yin$^\dagger$}

\begin{document}
\maketitle \bigskip

 \noindent {\bf Abstract} \ This paper studies
the normalized Ricci flow from a slight perturbation of the
hyperbolic metric on $\mathbb H^n$. It's proved that if the
perturbation is small and decays sufficiently fast at the infinity,
then the flow will converge exponentially fast to the hyperbolic
metric when the dimension $n>5$.

\section{Introduction}\label{S:Intro}

The Ricci flow of Hamilton evolves the metric of a Riemannian
manifold in the direction of an Einstein metric. There is a natural
question that if one starts from a small perturbation of an Einstein
metric, or without the a priori knowledge of the existence of an
Einstein metric, from a sufficiently (Ricci) pinched metric, can we
show the flow converges to (the) Einstein metric? The problem is
addressed by R. Ye in \cite{Ye}. Ye proved several theorems in this
direction in the case of closed Riemannian manifolds. Central to his
proof is a concept of \textit{stability}. On one hand, if the
solution to the Ricci flow remains stable, the $L^2$-norm of
traceless Ricci tensor decays exponentially. On the other hand, if
the solution remains pinched, it is stable. Since the traceless
Ricci tensor is almost the right hand side of the Ricci flow
equation, Ye was able to combine the above two observations to
conclude if the initial metric is sufficiently pinched then the
solution will remain so for any later time and converge to some
Einstein metric.

In this paper, we try to study a similar problem for complete noncompact manifolds. More precisely, we are concerned with the normalized Ricci flow
\begin{equation}
\pfrac{g_{ij}}{t}=-2(R_{ij}+(n-1)g_{ij}).
\label{eqn:NRF}
\end{equation}
Let $\mathbb H^n$ be the hyperbolic space with constant sectional curvature $-1$.
 Denote the hyperbolic metric by $g_{\mathbb H}$. A metric $g$ on $\mathbb H^n$ is said to be
  $\varepsilon$-hyperbolic for some positive $\varepsilon>0$ if
\begin{equation}
  (1-\varepsilon)g_{\mathbb H} \leq g\leq (1+\varepsilon)g_{\mathbb H}
  \label{eqn:comparable}
\end{equation}
and
\begin{equation}
  \abs{K(x,\sigma)+1}\leq \varepsilon
  \label{eqn:almosthyper}
\end{equation}
where $K(x,\sigma)$ is the sectional curvature of tangent plane
$\sigma$ at $x\in \mathbb H^n$. By a perturbation, we mean more than
just being $\varepsilon$-hyperbolic. Since we are working on a
complete noncompact manifold, it is natural to assume some
asymptotic condition at the infinity. For some $\delta>0$, the
metric $g$ is $\varepsilon$-hyperbolic of order $\delta$ if, in
addition to equation (\ref{eqn:comparable}),
\begin{equation}
  \abs{(K(x,\sigma)+1)e^{\delta d(x,x_0)}}\leq \varepsilon,
  \label{eqn:asymphyper}
\end{equation}
where $d(x, x_0)$ is the distance from $x$ to some
 fixed point $x_0$ with respect to the  metric $g.$

Our first result is
\begin{thm}
  \label{thm:first}
  For each $n\geq 3$ and $\delta>0$, there exists some $\varepsilon>0$ depending
  only on $\delta$ and $n$ such that the normalized Ricci flow starting
   from any $\varepsilon$-hyperbolic metric $g$ of order $\delta$ on
   $\mathbb H^n$ exists for all time and converges exponentially fast to some Einstein metric.
\end{thm}

The basic idea is more or less the same as in \cite{Ye}. It is about the two inter-wining facts: first, as long as the solution remains close to the hyperbolic space, the analysis properties of $g(t)$, precisely the spectrum radius, resembles those of the hyperbolic space; second, we will show once we have a lower bound of the spectrum radius, $\abs{R_{ij}+(n-1)g_{ij}}$ decays exponentially so that if the initial norm is small, it has no chance of getting too far away from the hyperbolic metric $g_{\mathbb H}$. The difference between the method here and that in \cite{Ye} is that we use a point-wise estimate of $\abs{R_{ij}+(n-1)g_{ij}}$ instead of $L^2$ estimate of the traceless Ricci tensor, which enables us to handle the case of very small decay.

There is another advantage of a point-wise estimate of $\abs{R_{ij}+(n-1)g_{ij}}$, compared with a global $L^2$ estimate. It allows us to prove that the limiting Einstein metric is asymptotically hyperbolic of a certain degree if the initial metric is asymptotically hyperbolic. Combined with rigidity results of Shi and Tian \cite{ShiTian}, this implies our second result

\begin{thm}
  \label{thm:second} For $n>5$ and $\delta> 2$, there exists $\varepsilon>0$ depending
  only on $\delta$ and $n$ such that the normalized Ricci flow starting from
  any $\varepsilon$-hyperbolic metric of order $\delta$ on $\mathbb H^n$ converges
  exponentially fast to $g_{\mathbb H}$.
\end{thm}

For $n\leq 5$, we can not prove the decay result to justify the above result. For $n=3$,
since every Einstein metric is of constant sectional curvature, the result is still true in a sense.
The case $n=5$ is the critical case as can be seen from the proof below.
The authors do not know whether this is essential or just a technical problem.
However, it is clear the the approach in this paper does not work for the case $n=4$. The condition $\delta>2$ is
 necessary. Due to the result of Graham and Lee in \cite{GrahamLee}, there does exist Poincar\'{e}-Einstein metric of non-constant curvature as close to the hyperbolic metric $g_{\mathbb H}$ as one needs such that $\abs{K(x,\sigma)+1}$ decays like $e^{-2d(x,x_0)}$ near the infinity.

The organization of the paper is: In Section \ref{section:basic}, we
establish
 two lemmas which are true on hyperbolic metric and which are robust enough to be still
 true for $\varepsilon$-hyperbolic metrics. In Section \ref{section:estimate}, the key estimate of this paper is proved. In the final section, we discuss the proof of Theorem \ref{thm:first} and Theorem \ref{thm:second}.

\bigskip

 {\bf Acknowledgements}: The first author  would like to
thank Professor P. Lu and J. Qing for very helpful discussions. Both
authors would like to thank Professor W. Y. Ding for warm
encouragements over the past several years.
\section{Basic facts on hyperbolic metric}
\label{section:basic}

In this section, we will show two basic results about analysis on
$\varepsilon$-hyperbolic metric. They follow easily from the
corresponding results on the hyperbolic metric. One of them is just
Lemma 5.4 in John Lee's paper \cite{Lee}. The other is Lemma 7.12 in
the same paper for the hyperbolic metric itself instead of
asymptotically hyperbolic metric.

\begin{lem}
  \label{lem:convolution}
  Let $g$ be an $\varepsilon$-hyperbolic metric on $\mathbb H^n$. For $a,b\in \Real$
  and $a+b>n-1$, $a>b$, there exists a constant $C$ depending only on $n,a,b$ such that
  for all $x,y\in \mathbb H^n$,
  \begin{equation*}
    \int_{\mathbb H^n} e^{-a d(x,z)}e^{-b d(y,z)}dV_z\leq C e^{-b d(x,y)}.
  \end{equation*}
\end{lem}

For the proof, see Lemma 5.4 in \cite{Lee}. Please note that the
dimension of the hyperbolic space is $n+1$ in \cite{Lee} while we
use $n$. Moreover, we can make the constant uniform with respect to
$\varepsilon$ when $\varepsilon$ goes to zero.

It's well known that the bottom of the spectrum of $\mathbb H^n$ is
$\frac{(n-1)^2}{4}$. It can be characterized by
\begin{equation*}
  \inf \frac{\int_{\mathbb H^n} \abs{\nabla f}^2 dV}{\int_{\mathbb H^n}f^2\, dV}
\end{equation*}
where the infimum is taken for {all smooth $f$ with compact
support}. Denote by $C_1(\varepsilon)$ by some constant depending
only on $\varepsilon$ and
\begin{equation*}
  \lim_{\varepsilon\to 0} C_1(\varepsilon)=0.
\end{equation*}
The same convention applies to all constants $C_i(\varepsilon)$
throughout this paper. For an $\varepsilon$-hyperbolic metric $g$,
we have, for compactly supported function or function with fast
decay such that the integration by parts can be justified,
\begin{equation}\label{eqn:poincare}
  \int_{\mathbb H^n} \abs{\nabla f}^2 dV_g \geq \Big(\frac{(n-1)^2}{4}-C_1(\varepsilon)\Big) \int_{\mathbb H^n} f^2 dV_g.
\end{equation}

Let $\xi$ be a traceless symmetric two tensor. There is also a bottom of the spectrum of the Laplacian
\begin{equation*}
  \tilde{\triangle} \xi= \sum_i \nabla_{e_i}\nabla_{e_i} \xi,
\end{equation*}
where $e_i$ are orthonormal frames. To obtain sharp constants in our theorems,
 we need the following lemma, which provides the sharp estimate for traceless symmetric two tensors.
\begin{lem}\label{lem:twotensor}
  For any metric $\varepsilon$-hyperbolic metric $g$, there exists a positive constant $C_2(n,\varepsilon)$ such that for each $\lambda\leq \frac{(n-1)^2}{4}+2-C_2(\varepsilon)$, the following inequality is true for $\xi$ with compact support,
\begin{equation*}
  \int_{\mathbb H^n} \abs{\nabla \xi}^2 dV_g\geq \lambda \int_{\mathbb H^n} \abs{\xi}^2 dV_g.
\end{equation*}
\end{lem}

\begin{proof}
  The proof runs parallel to that of Lemma 7.12 in \cite{Lee}. We need
  to be more careful to make sure the estimate is true globally instead
   of asymptotically. The idea of the proof is to take a symmetric two tensor as a $E$-valued $1$-form ($E=T^*\mathbb H$). For any Riemannian manifold $(M,g)$, one may define the induced connection
  \begin{equation*}
    D:C^\infty(M,\Lambda^q E)\to C^\infty(M;\Lambda^{q+1}E)
  \end{equation*}
  and the covariant Laplace-Beltrami operator on $E$-valued forms,
  \begin{equation*}
    \triangle=DD^*+D^*D,
  \end{equation*}
  where $D^*$ is the formal adjoint of $D$.
One can then prove by direct computation,
\begin{lem}
    \label{lem:sevennine}
  (Lemma 7.9 in \cite{Lee}) For any smooth compactly supported section $\xi$ of $\Lambda^q E$, and any positive $C^2$ function $\varphi$ on $M$, the following integral formula holds:
  \begin{equation}
    (\xi,\triangle \xi)\geq \int_M \langle \xi, (-\varphi^{-1}\triangle \varphi+2H(\log \varphi))\xi \rangle \,dV.
  \end{equation}
  Here $\langle \cdot,\cdot \rangle$ is the induced inner product of tensor bundles and $(\cdot,\cdot)$ is $\int_M \langle\cdot,\cdot\rangle dV$, $H(u)\xi$ is defined by $H(u)\xi=u_{ij}e^i\wedge(e^j\vee
\xi)$ and operator $\vee$ is defined to be the adjoint of
$\wedge$.
\end{lem}

\begin{rem}
  In this paper, we use a sign convention of Laplacian of functions different from \cite{Lee}.
\end{rem}
The above lemma is true for all complete manifolds. Now, let $g$ be
some $\varepsilon$-hyperbolic metric on $\mathbb H^n$. We will show
that there exists some $C_3(\varepsilon)$ such that for any
compactly supported traceless symmetric two tensor $\xi$,
\begin{equation}\label{eqn:aa}
  (\xi,\triangle \xi)\geq \frac{(n-3)^2}{4}(1-C_3(\varepsilon))(\xi,\xi).
\end{equation}
Let us pretend that (\ref{eqn:aa}) is proved. The next ingredient of the
proof is a Weitzenb\"{o}ck formula,
\begin{equation*}
  \triangle \xi=\tilde{\triangle} \xi +R_{kj}\xi_{ik} +R_{ikjl}\xi_{kl}.
\end{equation*}
Since $g$ is $\varepsilon$-hyperbolic,
\begin{equation*}
  \abs{R_{ikjl}-(-g_{il}g_{jk}+g_{ij}g_{kl})}\leq C_4(\varepsilon).
\end{equation*}
Therefore,
\begin{eqnarray*}
  \int_{\mathbb H^n} \abs{\nabla \xi}^2 dV &=& (\xi,\tilde{\triangle} \xi) \\
  &=& (\xi,\triangle \xi)-R_{kl}\xi_{ik}\xi_{jl}-R_{ikjl}\xi_{kl}\xi_{ij} \\
  &\geq& (\frac{(n-3)^2}{4}(1-C_3(\varepsilon))+(n-1)+1 -nC_4(\varepsilon)) (\xi,\xi)\\
  &=& (\frac{(n-1)^2}{4}+2-C_2(\varepsilon)) (\xi,\xi).
\end{eqnarray*}

It remains to prove (\ref{eqn:aa}). For any fixed traceless symmetric two tensor $\xi$, let $z\in M$ be some point such that $B(z,1)$ lies outside of the support of $\xi$. Set
\begin{equation*}
\varphi(x)=e^{-\frac{n-3}{2}r},
\end{equation*}
where $r=d(x,z)$. We make this particular choice of $z$ and $\varphi$ for two purposes. One is to make sure $\varphi$ is $C^2$ on the support set of $\xi$ so that Lemma \ref{lem:sevennine} applies. The other is to simplify some technical point, which will be clear later. We should then be careful that $C_3(\varepsilon)$ does not depend on this choice.

By Hessian comparison, we have
\begin{equation*}
Hess(r)\leq \sqrt{1+\varepsilon}\frac{\cosh
\sqrt{1+\varepsilon}\,r}{\sinh \sqrt{1+\varepsilon}\;r}(g-dr\otimes
dr).
\end{equation*}
Therefore,
\begin{eqnarray*}
(\log \varphi)_{ij}&=& -\frac{n-3}{2}r_{ij}\\
&\geq& -\frac{n-3}{2}\sqrt{1+\varepsilon}\frac{\cosh
\sqrt{1+\varepsilon}\;r}{\sinh
\sqrt{1+\varepsilon}\;r}(g_{ij}-r_ir_j).
\end{eqnarray*}
\begin{eqnarray}
\nonumber
\langle 2H(\log \varphi)\xi,\xi\rangle &\geq&
-(n-3)\sqrt{1+\varepsilon}\frac{\cosh \sqrt{1+\varepsilon}\;r}{\sinh \sqrt{1+\varepsilon}\;r}[\langle \xi,\xi\rangle -\langle dr\wedge dr \vee\xi,\xi\rangle ]\\
\nonumber
&=& -(n-3)\sqrt{1+\varepsilon}\frac{\cosh\sqrt{1+\varepsilon}\;r}{\sinh \sqrt{1+\varepsilon}\;r}[\langle \xi,\xi\rangle -\langle dr\vee \xi,dr\vee\xi\rangle] \\
&\geq&
-(n-3)\sqrt{1+\varepsilon}\frac{\cosh\sqrt{1+\varepsilon}\;r}{\sinh\sqrt{1+\varepsilon}\;r}\langle
\xi,\xi\rangle
  \label{eqn:bb}
\end{eqnarray}

By Laplacian comparison, we have
\begin{equation*}
\triangle r\geq (n-1)\sqrt{1-\varepsilon}\frac{\cosh
\sqrt{1-\varepsilon}\;r}{\sinh \sqrt{1-\varepsilon}\;r}.
\end{equation*}
Then,
\begin{eqnarray}\nonumber
-\varphi^{-1}\triangle \varphi&=& -\frac{(n-3)^2}{4}+\frac{n-3}{2}\triangle r \\
&\geq&
-\frac{(n-3)^2}{4}+\frac{(n-3)(n-1)}{2}\sqrt{1-\varepsilon}\frac{\cosh\sqrt{1-\varepsilon}\;r}{\sinh\sqrt{1-\varepsilon}\;r}.
  \label{eqn:cc}
\end{eqnarray}
Combining equation (\ref{eqn:bb}) and (\ref{eqn:cc}) and Lemma \ref{lem:sevennine}, we have
\begin{eqnarray*}
(\xi,\triangle \xi)&\geq& \int_M (-\frac{(n-3)^2}{4}-(n-3)\sqrt{1+\varepsilon}\frac{\cosh\sqrt{1+\varepsilon}\;r}{\sinh\sqrt{1+\varepsilon}\;r}\\
&&+\frac{(n-3)(n-1)}{2}\sqrt{1-\varepsilon}\frac{\cosh\sqrt{1-\varepsilon}\;r}{\sinh\sqrt{1-\varepsilon}t})\langle
\xi,\xi\rangle dV
\end{eqnarray*}

\begin{lem} For any $\delta>0$, there exists some positive constant $\varepsilon_0(n, \delta)$ such that
\begin{equation*}
-\frac{(n-3)^2}{4}-(n-3)\sqrt{1+\varepsilon}\frac{\cosh\sqrt{1+\varepsilon}\;r}{\sinh
\sqrt{1+\varepsilon}\;r}
+\frac{(n-3)(n-1)}{2}\sqrt{1-\varepsilon}\frac{\cosh\sqrt{1-\varepsilon}\;r}{\sinh\sqrt{1-\varepsilon}\;r}\geq
\frac{(n-3)^2}{4}-\delta
\end{equation*}
holds for $r\in [1,\infty)$ and $\varepsilon$ smaller than
$\varepsilon_0(n, \delta)$.
\end{lem}
\begin{rem}
We have assumed that $B(z,1)$ is outside of the support of $\xi$. This explains why it suffices to prove the lemma for $r\in [1,\infty)$ only, which is a simplification.
\end{rem}
\begin{proof}
Given $\delta$, there is some big $R>0$ such that $r>R$ implies
$\frac{\cosh
\sqrt{1\pm\varepsilon}\;r}{\sinh\sqrt{1\pm\varepsilon}\;r}$ are very
close to $1$. Therefore, the inequality is true if $\varepsilon$ is
small.

For each $r\in [1,R]$, when $\varepsilon$ goes to $0$, the limit of the right hand side is
\begin{equation*}
-\frac{(n-3)^2}{4}-(n-3)\frac{\cosh r}{\sinh r}+\frac{(n-3)(n-1)}{2}\frac{\cosh r}{\sinh r}\geq \frac{(n-3)^2}{4}.
\end{equation*}
The limit is a continuous function. Hence by Dini's theorem in calculus, the convergence is uniform with respect to $r\in [1,R]$.
\end{proof}

This finishes our proof of equation (\ref{eqn:aa}) and the Lemma \ref{lem:twotensor}.

\end{proof}

\section{Pointwise estimate of $\abs{R_{ij}+(n-1)g_{ij}}$}
\label{section:estimate}

This section contains the proof of an estimate of $\abs{R_{ij}+(n-1)g_{ij}}$, which is the most important part in the proof of the main theorems. The proof of this estimate follows the method of Grigor'yan in \cite{Grigor}. In that paper, Grigor'yan proved a sharp estimate (Proposition 5.1) of heat kernel on complete manifolds with a positive spectrum radius. In fact, the authors tried to modify his proof and use it for our purpose. Fortunately, it turns out that it is much easier to use the method instead of the result. In a sense, we extend his result from estimating functions to tensors. Moreover, the Lemma \ref{lem:twotensor} comes in to provide sharp constants ($n$ and $\delta$) in the main theorems.

Let $g(t), t\in [0,T]$ be a solution to the normalized Ricci flow (\ref{eqn:NRF})
where $T>\eta$ for some positive constant $\eta$. Assume that there exists some $\varepsilon>0$ such that for each $t$, $g(t)$ is $\varepsilon$-hyperbolic. Note that in this section, we \textit{do not} assume that the derivatives of curvature tensor are bounded for $t\in (0,T]$, although we do know that for each time slice derivatives of any order are bounded. We need this fact later in this section to justify an integration by parts. For simplicity, set
\begin{equation*}
h_{ij}(t)=R_{ij}+(n-1)g_{ij}.
\end{equation*}

The main estimates in this section are the following two lemmas:

\begin{lem}\label{lem:point}For any $x\in \mathbb H^n$ and $t>\eta$, we have
\begin{equation}\label{eqn:finalpoint}
  \int_\eta^T \abs{h_{ij}}(x,t)dt
  \leq C(\eta, \varepsilon) \left( \int_{\mathbb H^n} \abs{h_{ij}}^2(y,0)\exp
   (-(2\sqrt{\lambda}-C_{13}(\varepsilon))d_0(y,x))dy \right)^{1/2},
\end{equation}
where $\lambda$ is any positive constant no bigger than
$\frac{(n-1)^2}{4}-\max \{C_1(\varepsilon),C_2(\varepsilon)\}$.
\end{lem}

\begin{lem}\label{lem:parabolic}
    For any $x\in \mathbb H^n$ and $t>2\eta,$ we have
\begin{equation}\label{eqn:final}
    \int_{2\eta}^{T}\max_{B_0(x,\sqrt{\eta})\times [t-\eta,t]}\abs{h_{ij}}dt
  \leq C(\eta, \varepsilon) \left( \int_{\mathbb H^n} \abs{h_{ij}}^2(y,0)\exp
   (-(2\sqrt{\lambda}-C_{13}(\varepsilon))d_0(y,x))dy \right)^{1/2},
\end{equation}
where $\lambda$ is any positive constant no bigger than
$\frac{(n-1)^2}{4}-\max \{C_1(\varepsilon),C_2(\varepsilon)\}$.
\end{lem}
\begin{rem}
    Lemma \ref{lem:parabolic} is prepared for estimating the derivatives of $h_{ij}$.
\end{rem}

\begin{proof} (For both lemmas.)
It is straight forward to compute the evolution equations for $h_{ij}$ and $\abs{h_{ij}}^2$:
\begin{equation}
\pfrac{}{t}h_{ij}=\triangle h_{ij}-2R_{ipjq}h_{pq}-2h_{ip}h_{pj}
\label{eqn:hij}
\end{equation}
\begin{equation}
\pfrac{}{t}\abs{h_{ij}}^2=\triangle \abs{h_{ij}}^2 -2\abs{h_{ij,k}}^2-4R_{ipjq}h_{ij}h_{pq}-4h_{ip}h_{pj}h_{ij}.
\label{eqn:normofh}
\end{equation}
Since $g(t)$ is $\varepsilon$-hyperbolic,
\begin{equation*}
\pfrac{}{t}\abs{h_{ij}}^2\leq \triangle \abs{h_{ij}}^2 + c \abs{h_{ij}}^2,
\end{equation*}
where $c$ is some universal constant. For each $t>\eta$ and any $x\in M$, consider a parabolic ball $B_0(x,\sqrt{\eta})\times [t-\eta,t]$. Here by $B_0$ we mean geodesic ball measured by $g(0)$. Recall that $g(t)$ and $g_{\mathbb H}$ are comparable by a constant $1+\varepsilon$. Therefore, we have uniform Sobolev inequality on $B_0(x,\sqrt{\eta})$ for each $g(s),s\in [t-\eta,t]$. The standard Moser iteration gives
\begin{equation}
\abs{h_{ij}}^2(x,t)\leq C(\eta) \int_{t-\eta/2}^t\int_{B_0(x,\sqrt{\eta}/2)}\abs{h_{ij}}^2(y,x)dyds.
\label{eqn:Moser}
\end{equation}

The following lemma is a direct consequence of $g(t)$ being
$\varepsilon$-hyperbolic.

\begin{lem}
There exists a constant $C_5(\varepsilon)$ (which vanishes if $\varepsilon$ goes to zero) such that
\begin{equation*}
\xi(y,s)=-\frac{d_0^2(y)}{(2+C_5(\varepsilon))(t-s)}
\end{equation*}
satisfies
\begin{equation*}
\xi_s+\frac{1}{2}\abs{\nabla \xi}^2\leq 0
\end{equation*}
for $s<t$. Here $d_0(y)$ is a distance function with respect to
$g(0)$ and the norm and $\nabla$ are those of $g(s)$.
\end{lem}

Now, let $d_0(y)$ be the distance to $B_0(x,\sqrt{\eta}/2)$ measured with respect to $g(0)$ and $\xi$ be as in the previous lemma. Since $\xi(y,s)\equiv 1$ for $y\in B_0(x,\sqrt{\eta}/2)$,
\begin{equation}\label{eqn:jjj}
\abs{h_{ij}}(x,t)^2\leq C\int_{t-\eta/2}^t \int_{\mathbb H^n}
\abs{h_{ij}}^2(y,s)e^{\xi} dyds.
\end{equation}
The integrability is not a problem since when $s<t$, $e^\xi$ decays very fast at the infinity and $\abs{h_{ij}}^2$ is always bounded. Set
\begin{equation*}
I(s)=\int_{\mathbb H^n}\abs{h_{ij}}^2(y,s)e^\xi dy.
\end{equation*}

Then,
\begin{eqnarray*}
\frac{d I}{ds}(s)&=& \frac{d}{ds}\int_{\mathbb H^n} \abs{h_{ij}}^2 e^{\xi}dy \\
&\leq& \int_{\mathbb H^n} 2\langle \pfrac{h_{ij}}{s},h_{ij}\rangle e^{\xi}+\abs{h_{ij}}^2e^\xi\xi_s +\abs{h_{ij}}^2e^{\xi}C_6(\varepsilon)dy \\
&=& \int_{\mathbb H^n} 2\langle \triangle h_{ij}-2R_{ipjq}h_{pq}-2h_{ip}h_{pj},h_{ij}\rangle e^{\xi} +\abs{h_{ij}}^2 e^\xi \xi_s +\abs{h_{ij}}^2 e^\xi C_6(\varepsilon) dy  \\
&\leq& \int_{\mathbb H^n} 2\langle \triangle h_{ij},h_{ij}\rangle
e^{\xi} -4R_{ipjq}h_{pq}h_{ij}e^\xi +\abs{h_{ij}}^2 e^\xi\xi_s
+\abs{h_{ij}}^2 e^\xi C_7(\varepsilon)dy.
\end{eqnarray*}
Here we use $\varepsilon$-hyperbolic in the second line, the
evolution equation of $h_{ij}$ in the third and the fact
$\abs{h_{ij}}\leq n\varepsilon$ in the last.

Consider the following divergence term, whose integration vanishes. (Of cause, this needs justification. $e^{\xi}$ decays very fast at infinity while $h_{ij}$ together with its derivatives remain bounded.)
\begin{eqnarray*}
\left( \langle \nabla_k h_{ij},h_{ij}\rangle e^{\xi} \right)_{,k} &=& \langle \triangle h_{ij},h_{ij}\rangle e^\xi +\abs{\nabla_kh_{ij}}^2 e^{\xi} +\nabla_kh_{ij}h_{ij}e^{\xi}\xi_k .
\end{eqnarray*}

\begin{eqnarray*}
\int_{\mathbb H^n} 2\triangle h_{ij}h_{ij}e^\xi +\abs{h}^2 e^{\xi}\xi_s dy&=& \int_{\mathbb H^n} -2\abs{\nabla_k h_{ij}}^2 e^\xi -2\nabla_k h_{ij}h_{ij}\xi_k e^\xi +\abs{h_{ij}}^2 e^\xi \xi_s \\
&=& \int_{\mathbb H^n} -\frac{1}{2}e^\xi\abs{2\nabla_k h_{ij}+\xi_k h_{ij}}^2 +\abs{h_{ij}}^2 e^{\xi}(\xi_s+\frac{1}{2}\abs{\nabla\xi}^2)\\
&\leq& -2\int_{\mathbb H^n} \abs{\nabla_k (e^{\xi/2}h_{ij})}^2
\end{eqnarray*}
Since $g(t)$ is $\varepsilon$-hyperbolic, we have
\begin{equation*}
\abs{R_{ipjq}-(g_{ij}g_{pq}-g_{iq}g_{jp})}\leq C_8(\varepsilon).
\end{equation*}

\begin{eqnarray*}
\frac{dI}{ds}(s)&\leq& \int_{\mathbb H^n} -2\abs{\nabla_k (e^{\xi/2}h_{ij})}^2 +4(-g_{ij}g_{pq}+g_{iq}g_{jp})h_{ij}h_{pq}e^{\xi} +C_9(\varepsilon) I(s) \\
&=&\int_{\mathbb H^n} -2\abs{\nabla_k (e^{\xi/2}h_{ij})}^2 -4
\mbox{tr}(h)^2e^{\xi} + 4\abs{h_{ij}}^2e^{\xi} +C_9(\varepsilon)
I(s)
\end{eqnarray*}
For abbreviation, we write $\mbox{tr}$ for
$\mbox{tr}(e^{\xi/2}h_{ij})$ and $\tilde{h}_{ij}$ for the traceless
part of $e^{\xi/2}h_{ij}$, that is
\begin{equation*}
e^{\xi/2}h_{ij}=\frac{\mbox{tr}}{n}g_{ij}+\tilde{h}_{ij}.
\end{equation*}

\begin{equation*}
\abs{e^{\xi/2}h_{ij}}^2=\frac{\mbox{tr}^2}{n}+\abs{\tilde{h}_{ij}}^2.
\end{equation*}

Moreover,
\begin{equation*}
\abs{\nabla_k(e^{\xi/2}h_{ij})}^2=\frac{1}{n}\abs{\nabla \mbox{tr}}^2+\abs{\nabla_k \tilde{h}_{ij}}^2.
\end{equation*}

\begin{eqnarray*}
\frac{dI}{ds}(s) &\leq& \int_{\mathbb H^n} -\frac{2}{n}\abs{\nabla \mbox{tr}}^2 -2\abs{\nabla \tilde{h}}^2 -4\mbox{tr}^2 +\frac{4}{n}\mbox{tr}^2 +4\abs{\tilde{h}}^2 +C_9(\varepsilon)I(s)\\
&\leq& \int_{\mathbb H^n} -\frac{2}{n}\abs{\nabla \mbox{tr}}^2 -2\abs{\nabla \tilde{h}}^2  +4\abs{\tilde{h}}^2 +C_9(\varepsilon)I(s)\\
&\leq& \int_{\mathbb H^n} -\frac{2\lambda}{n}\mbox{tr}^2 -2\lambda \abs{\tilde{h}}^2 +C_{10}(\varepsilon)I(s)\\
&\leq&- 2\lambda I(s) +C_{10}(\varepsilon)I(s).
\end{eqnarray*}
Here $\lambda$ is any positive constant no bigger than $\frac{(n-1)^2}{4}-\max \{C_1(\varepsilon),C_2(\varepsilon)\}$. We have used Lemma \ref{lem:twotensor} and the Poincar\'{e} inequality (\ref{eqn:poincare}).
By ODE comparison,
\begin{equation*}
I(s)\leq e^{-(2\lambda-C_{10}(\varepsilon))s}I(0).
\end{equation*}
Together with equation (\ref{eqn:jjj}), this implies
\begin{equation*}
\abs{h_{ij}}^2(x,t)\leq C(\eta) e^{-(2\lambda-C_{10}(\varepsilon))t}I(0).
\end{equation*}
Hence,
\begin{equation}\label{eqn:aaa}
  \abs{h_{ij}}^2(x,t)\leq C(\eta) \int_{\mathbb H^n} \abs{h_{ij}}^2(y,0)\exp \left( -\frac{d_0^2(y)}{(2+C_5(\varepsilon))t}-(2\lambda-C_{10}(\varepsilon))t \right) dy.
\end{equation}
This is a point-wise estimate on $h_{ij}$, from which one can know
the behavior of $h_{ij}$ for any $t>\eta$ and $x\in \mathbb H^n$. We
will be interested in an estimate of the accumulation of $h_{ij}$,
since it measures the change of $g_{ij}$.
\begin{equation*}
\abs{h_{ij}}^2(x,t)\leq
C(\eta)\exp(-C_{12}(\varepsilon)t)\int_{\mathbb
H^n}\abs{h_{ij}}^2(y,0)\exp\left(
-\frac{d_0^2(y)}{(2+C_5(\varepsilon))t}-(2\lambda-C_{11}(\varepsilon))t
\right)dy.
\end{equation*}
Due to the basic inequality $a+b\geq 2\sqrt{ab}$, there exists some
$C_{13}(\varepsilon)$ such that
\begin{equation}\label{eqn:basic}
\abs{h_{ij}}^2(x,t)\leq
C(\eta)\exp(-C_{12}(\varepsilon)t)\int_{\mathbb
H^n}\abs{h_{ij}}^2(y,0)\exp\left(
-(2\sqrt{\lambda}-C_{13}(\varepsilon))d_0(y) \right)dy.
\end{equation}
Recall that $d_0(y)$ is the distance to $B_0(x,\sqrt{\eta}/2)$. By
triangle inequality,
\begin{equation*}
d_0(y)>d_0(y,x)-\sqrt{\eta}/2.
\end{equation*}
Hence,
\begin{equation*}
\exp( -(2\sqrt{\lambda}-C_{13}(\varepsilon))d_0(y))\leq C(\eta) \exp( -(2\sqrt{\lambda}-C_{13}(\varepsilon))d_0(y,x)).
\end{equation*}
Then inequality (\ref{eqn:basic}) is still true if we replace $d_0(y)$ by $d_0(y,x)$. Taking the square root and integrating over time proves Lemma \ref{lem:point}.

By a similar argument,
\begin{equation*}
\max_{B_0(x,\sqrt{\eta})}\abs{h_{ij}}^2(\cdot,t)\leq
C(\eta)\exp(-C_{12}(\varepsilon)t)\int_{\mathbb
H^n}\abs{h_{ij}}^2(y,0)\exp\left(
-(2\sqrt{\lambda}-C_{13}(\varepsilon))d_0(y,x) \right)dy.
\end{equation*}
For each $t>2\eta$, the above inequality is true for $s\in [t-\eta,t]$. Similar argument in $t$ direction gives
\begin{equation}\label{eqn:prefinal}
\max_{B_0(x,\sqrt{\eta})\times [t-\eta,t]}\abs{h_{ij}}^2\leq
C(\eta)\exp(-C_{12}(\varepsilon)t)\int_{\mathbb
H^n}\abs{h_{ij}}^2(y,0)\exp\left(
-(2\sqrt{\lambda}-C_{13}(\varepsilon))d_0(y,x) \right)dy.
\end{equation}
Taking the square root and integrating over time, we have
\begin{equation*}
  \int_\eta^{T}\max_{B_0(x,\sqrt{\eta})\times [t-\eta,t]}\abs{h_{ij}}dt\leq
   C(\eta, \varepsilon) \left( \int_{\mathbb H^n} \abs{h_{ij}}^2(y,0)\exp
    (-(2\sqrt{\lambda}-C_{13}(\varepsilon))d_0(y,x))dy \right)^{1/2}.
\end{equation*}
This finishes the proof of Lemma \ref{lem:parabolic}.
\end{proof}
\section{Proof of the theorems}
\label{section:proof}

We can now show the proof of the main theorems. Let's look at the proof of Theorem \ref{thm:first}.
 The theorem assumes that the initial metric is some $\varepsilon$-hyperbolic metric of order $\delta$ for some $\delta>0$. Given this $\delta>0$, pick an $\varepsilon_1>0$ such that
\begin{equation*}
2\sqrt{\frac{(n-1)^2}{4}-\max\{C_1(\varepsilon_1),C_2(\varepsilon_1)}\}-C_{13}(\varepsilon_1)+2\delta>n-1.
\end{equation*}
For the meaning of $C_i$'s, see the previous section. This is possible because
\begin{equation*}
\lim_{\varepsilon\to 0}C_i(\varepsilon)=0.
\end{equation*}

Let $g(0)$ be any $\frac{1}{10}\varepsilon_1$-hyperbolic initial
metric. Due to a result of Shi\cite{ShiWX},
there exists a local solution to the normalized Ricci flow
(\ref{eqn:NRF}). By continuity, there exists some $\tau>0$ such that
for each $t\in [0,\tau]$, $g(t)$ is
$\frac{1}{2}\varepsilon_1$-hyperbolic. Let $T$ be the maximum number
such that $g(t)$ remains $\varepsilon_1-$hyperbolic for $t\in
[0,T]$. By another result of Shi in the same paper, there exists constant $C(k,\tau)$, such that the $k-$th derivatives
of curvature tensor of $g(t)$ are bounded by $C(k,\tau)$ uniformly
for $t\in [\frac{\tau}{2},T]$.

We can now apply the result of Section \ref{section:estimate} to the
flow $g(t)$ with $t\in [0,T]$ and $\eta=\tau/2$. Since $g(0)$ is
$\varepsilon$-hyperbolic of order $\delta$, we have
\begin{equation*}
\abs{h_{ij}(y,0)}\leq C\varepsilon e^{-\delta d_0(y,x_0)}.
\end{equation*}
By Lemma \ref{lem:point},
\begin{equation}
  \int_{\frac{\tau}{2}}^{T}\abs{h_{ij}(x,t)}dt\leq C(\varepsilon_1) \varepsilon \left( \int_{\mathbb H^n} \exp(-2\delta d_0(y,x_0))\exp (-(2\sqrt{\lambda}-C_{13}(\varepsilon_1))d_0(y,x))dy \right)^{1/2}.
\end{equation}
By our choice of $\varepsilon_1$, set $\lambda=\frac{(n-1)^2}{4}-\max\{C_1(\varepsilon_1),C_2(\varepsilon_1)\}$, then integral in the right hand side in the above inequality is finite. Hence,
\begin{equation}\label{eqn:11}
\int_{\frac{\tau}{2}}^T \abs{h_{ij}}(x,t)dt \leq C\varepsilon.
\end{equation}
Since $g(\frac{\tau}{2})$ is $\frac{1}{2}\varepsilon_1$-hyperbolic,
\begin{equation*}
(1-\frac{1}{2}\varepsilon_1)g_{\mathbb H^n}\leq
g(\frac{\tau}{2})\leq (1+\frac{1}{2}\varepsilon_1)g_{\mathbb H^n}.
\end{equation*}
The normalized Ricci flow equation
\begin{equation*}
\pfrac{g_{ij}}{t}=-2h_{ij}
\end{equation*}
together with equation (\ref{eqn:11}) implies one can choose $\varepsilon$ small so that
\begin{equation*}
(1-\frac{3}{4}\varepsilon_1)g_{\mathbb H^n}\leq g(t)\leq
(1+\frac{3}{4}\varepsilon_1)g_{\mathbb H^n}
\end{equation*}
for each $t\in [\frac{\tau}{2},T]$.

If we can show
\begin{equation}
\abs{K_{g(t)}(x,\sigma)+1}\leq \frac{3}{4}\varepsilon_1
\label{eqn:bbb}
\end{equation}
for each $t\in [\tau,T]$, then $g(t)$ is $\frac{3}{4}\varepsilon_1-$hyperbolic. This implies $T=\infty$ by its definition. The proof of equation (\ref{eqn:bbb}) involves higher order derivative estimates of $h_{ij}$, because
\begin{equation*}
\pfrac{}{t}{R_{ijk}}^l=-g^{lp}\left\{
\begin{array}[]{c}
\nabla_i\nabla_j h_{kp} +\nabla_i\nabla_k h_{jp}-\nabla_i\nabla_p h_{jk} \\
-\nabla_j\nabla_i h_{kp}-\nabla_j\nabla_k h_{ip}+\nabla_j\nabla_p h_{ik}
\end{array}
\right\}.
\end{equation*}
Recall that $h_{ij}$ satisfies
\begin{equation*}
\pfrac{}{t}h_{ij}=\triangle h_{ij}-2R_{ipjq}h_{pq}-2h_{ip}h_{pj}.
\end{equation*}

For each $t\in (\tau,T]$ and $x\in \mathbb H^n$, consider a
parabolic ball $B_0(x,\sqrt{\tau/2})\times [t-\tau/2,t]$ where the
radius is measured by $g_0$. Due to the assumption that $g(t)$ is
$\varepsilon_1$-hyperbolic and the derivatives of curvature tensor
are uniformly bounded by $C(k,\tau)$ for $t\in [\tau/2,T]$, one can choose
a local coordinate system on $B_0(x,\sqrt{\tau/2})$, for example the
harmonic coordinates with respect to $g(t-\tau/2)$, such that $g_{ij}(x,t)$
together with its derivatives are bounded in the parabolic ball
$B_0(x,\sqrt{\tau/2})\times [t-\tau/2,t]$.

We then apply the standard parabolic estimate to equation (\ref{eqn:hij}) in this parabolic ball to show
\begin{equation*}
\abs{\pfrac{h_{ij}}{x_k}}(x,t),\abs{\frac{\partial^2 h_{ij}}{\partial x_k \partial x_l}}(x,t)\leq C \max_{B_0(x,\sqrt{\tau/2})\times [t-\tau/2,t]} \sum_{p,q}\abs{h_{pq}}.
\end{equation*}
This implies
\begin{equation}\label{eqn:high}
\abs{\nabla\nabla h_{ij}}(x,t)\leq C \max_{B_0(x,\sqrt{\tau/2})\times [t-\tau/2,t]} \sum_{p,q}\abs{h_{pq}}.
\end{equation}
The method of choosing harmonic coordinates and the parabolic estimate involved here are rather routine but lengthy. We move the detail to the appendix.

Instead of estimating $\abs{K(x,\sigma)+1}$ for any $x$ and
$\sigma$, we consider
$\abs{R_{ipjq}-(g_{ij}g_{pq}-g_{iq}g_{jp})}$.

\begin{eqnarray*}
\pfrac{}{t}\abs{R_{ipjq}-(g_{ij}g_{pq}-g_{iq}g_{jp})}(x,t)&\leq& C(\abs{\pfrac{}{t}R_{ipjq}}+\abs{\pfrac{}{t}g_{ij}}) \\
&\leq& C\max_{B_0(x,\sqrt{\tau/2})\times [t-\tau/2,t]} \sum_{p,q}\abs{h_{pq}}
\end{eqnarray*}
Due to Lemma \ref{lem:parabolic} ($\eta=\tau/2$) and our choice of
$\varepsilon_1$, if $g(0)$ is $\varepsilon$-hyperbolic of order
$\delta$,
\begin{equation*}
\int_{\tau}^T \pfrac{}{t}\abs{R_{ipjq}-(g_{ij}g_{pq}-g_{iq}g_{jp})}(x,s)ds\leq C(\varepsilon_1)\varepsilon.
\end{equation*}
Therefore, equation (\ref{eqn:bbb}) is true if we choose
$\varepsilon$ small. Hence $T=\infty$, that is the solution will
exist for all time and remain $\varepsilon_1$-hyperbolic for ever.
It is now obvious from equation (\ref{eqn:prefinal}) that the
solution will converge to some Einstein metric. This finishes the
proof of Theorem \ref{thm:first}.

\medskip

We can now turn to the proof of Theorem \ref{thm:second}. Here is a result of Shi and Tian \cite{ShiTian} on the rigidity of hyperbolic space.

\begin{thm}
  Suppose that $(X^n,g), n\geq 3$ and $n\ne 4$ is an ALH manifold of order $\alpha(\alpha>2)$, $K\leq 0$ and $Ric(g)\geq -(n-1)g$, then $(X^{n},g)$ is isometric to $(\mathbb H^n, g_{\mathbb H^n})$.
\end{thm}
A complete noncompact Riemannian manifold is called by Shi and Tian an ALH manifold of order $\alpha$ if $\abs{K(x,\sigma)+1}=O(e^{-\alpha d_g(x,o)})$ for some fixed point $o$.

Suppose now $n> 5$ and $\delta>2$, by Theorem \ref{thm:first}, there
exists some $\varepsilon>0$ such that the normalized Ricci flow from
an $\varepsilon$-hyperbolic metric of order $\delta$ converges to an
Einstein metric $g_\infty$. In the proof, we can see that the
solution remains $\varepsilon_1$-hyperbolic, hence the sectional
curvature of the limit is negative. The condition $Ric(g_\infty)\geq
-(n-1)g_\infty$ is automatically satisfied by the Einstein metric.
Therefore to prove Theorem \ref{thm:second}, it suffices to show
\begin{equation*}
  \abs{K_\infty(x,\sigma)+1}\leq C e^{-\alpha d_\infty(x,x_0)}
\end{equation*}
for some $\alpha>2$. As before, we study
$\abs{R_{ipjq}-(g_{ij}g_{pq}-g_{iq}g_{jp})}$ instead of
$K(x,\sigma)+1$. Since the initial metric is
$\varepsilon$-hyperbolic of order $\delta$, we know there exists
some $C$ such that
\begin{equation}\label{eqn:ddd}
  \abs{R_{ipjq}-(g_{ij}g_{pq}-g_{iq}g_{jp})}(x,0)\leq C e^{-\delta d_0(x,x_0)}.
\end{equation}

The method we used above is not readily applicable here, because the higher order derivative estimate (\ref{eqn:high}) works for $t$ bigger than some fixed positive constant. However, we need a bound for $\pfrac{}{t}{R_{ijk}}^l$ immediately after $t=0$. Therefore, we need the following lemma, which is an application of maximum principle on complete manifold. It shows that if $\abs{K_{g_0}(x,\sigma)+1}$ decays at the order $\delta$ with respect to $g(0)$, then $\abs{K_{g(t)}(x,\sigma)+1}$ decays at the same order $\delta$ with respect to $g(t)$ for $t>0$.

\begin{lem}
  \label{lem:maximum}
  Let $(M,g)$ be a complete Riemannian manifold such that equation (\ref{eqn:ddd}) holds.
Then for any $t>0$ as long as the solution exists, we
have
\begin{equation*}
  \abs{R_{ipjq}-(g_{ij}g_{pq}-g_{iq}g_{jp})}(x,t)\leq C(t) e^{-\delta d_t(x,x_0)}.
\end{equation*}
\end{lem}

\begin{proof}We follow the argument in \cite{MD} and \cite{EK}. Set
$$Q_{ipjq}=R_{ipjq}-(g_{ij}g_{pq}-g_{iq}g_{jp}).$$
Direct calculation shows that
\begin{equation*}
  \frac {\partial}{\partial t}Q=\Delta Q+Q*Rm+2(n-1)Q,
\end{equation*}
where $Q*Rm$ denotes a sum of contractions of $Q$ and the curvature
tensor $Rm.$ Since the curvature of the initial metric is bounded,
if there is a positive $t_0>0$ such that the solution is defined for $t\in [0,t_0]$, then there exists constant $C$,
$$|Rm|(t)\leq C, \quad t\in [0, t_0].$$
By the comparison theorem, the derivatives of
the distance function are uniformly bounded:
$$|\nabla d_t(x, x_0)|+|\nabla^2 d_t(x, x_0)|\leq C,$$
where we can smooth $d_t$ at $x_0$ such that its derivatives are
uniformly bounded. Since we only consider the asymptotic behavior
when $x$ go to infinity, the non-smoothness of the distance function
at $x_0$ can be ignored. Combining the above estimates, we have
$$\frac {\partial}{\partial t}(e^{\delta d_t(x, x_0)}|Q|)\leq \Delta(e^{\delta d_t(x, x_0)}|Q|)+
\textbf{a}\cdot \nabla(e^{\delta d_t(x, x_0)}|Q|)+be^{\delta d_t(x,
x_0)}|Q| $$ where $\textbf{a}$ is a vector with $|\textbf{a}|\leq C$
and $b$ is a constant.  Note that $Q$ satisfies (\ref{eqn:ddd}), by
the maximum principle on complete manifolds in \cite{EK} we have
$$e^{\delta d_t(x, x_0)}|Q|(x, t)\leq e^{bt}\max_{\mathbb H^n}(e^{\delta d_0(x, x_0)}|Q|(x, 0))\leq C e^{bt},
\quad t\in [0, t_0].$$ Thus, the lemma is proved.
\end{proof}

\begin{cor}\label{cor}
  Let $(M,g)$ be a complete Riemannian manifold such that equation (\ref{eqn:ddd}) holds with $\delta>2$. Let $\tilde{\delta}$ be any constant in $(2,\delta)$. There exists some $\tau>0$ such that for each $t\in[0,\tau]$,
    \begin{equation*}
        \abs{R_{ipjq}-(g_{ij}g_{pq}-g_{iq}g_{jp})}(x,t)\leq Ce^{-\tilde{\delta} d_0(x,x_0)}.
    \end{equation*}
\end{cor}

Given this corollary, we can argue as before. The key point is Lemma \ref{lem:parabolic} and Lemma \ref{lem:convolution}. In equation
(\ref{eqn:prefinal}),

\begin{equation*}
  \abs{h_{ij}}^2(y,0)\leq C e^{-2\delta d_0(y,x_0)}
\end{equation*}
and
\begin{equation*}
  \exp \left( -(2\sqrt{\lambda}-C_{13}(\varepsilon))d_0(y,x) \right)\leq C e^{-(\sqrt{(n-1)^2-\max\{C_1(\varepsilon_1),C_2(\varepsilon_1)\}}-C_{13}(\varepsilon_1))d_0(y,x)}.
\end{equation*}
If $n>5$ and $\delta>2$, there exists $\alpha$ such that
\begin{equation*}
    2<\alpha<\tilde{\delta}
\end{equation*}
and
\begin{equation*}
  2\alpha<\sqrt{(n-1)^2-\max{C_1(\varepsilon_1),C_2(\varepsilon_1)}}-C_{13}(\varepsilon_1).
\end{equation*}
Lemma \ref{lem:convolution}  and equation (\ref{eqn:prefinal}) implies that for $t>\tau$,
\begin{equation*}
  \max_{B_0(x,\sqrt{\eta})\times [t-\eta,t]}\abs{h_{ij}}^2\leq C \exp(-C_{12}(\varepsilon_1)t)\exp(-2\alpha d_0(x,x_0)).
\end{equation*}
The same argument as before gives
\begin{equation*}
  \int_\tau^\infty \pfrac{}{t}\abs{R_{ipjq}-(g_{ij}g_{pq}-g_{iq}g_{jp})}(x,s)ds\leq C
  \exp(-\alpha d_0(x,x_0)).
\end{equation*}
Together with Corollary \ref{cor}, we have
\begin{equation*}
  \abs{K_\infty(x,\sigma)+1}\leq C e^{-\alpha d_\infty(x,x_0)}.
\end{equation*}
Here we used the fact that $g(0)$ and $g_\infty$ are dominated by
each other. The rigidity theorem of Shi and Tian implies that
$g_\infty$ is $g_{\mathbb H}$, which concludes the proof of Theorem
\ref{thm:second}.

\section{Appendix: Higher order estimates}
The purpose of this appendix is to get higher order derivative estimate of $h_{ij}$ in the following equation.
\begin{equation}\label{hagain}
\pfrac{}{t}h_{ij}=\triangle h_{ij}-2R_{ipjq}h_{pq}-2h_{ip}h_{pj}
\end{equation}
We will prove a local estimate in the parabolic neighborhood $B_0(x,\sqrt{\tau/2})\times [t-\tau/2,t]$. The difficulty is that if we write equation (\ref{hagain}) in terms of coordinates, then the coefficients involves $\pfrac{g_{ij}}{x_k}$. For our purpose, we need to estimate second derivatives of $h_{ij}$ in terms of the $L^\infty$ norm. It is a routine technique in PDE to apply $L^p$ estimate and Sobolev embedding to get $C^\alpha$ H\"{o}lder norm of $h_{ij}$, then apply H\"{o}lder estimate. In order that the above estimates work, we need to control the $C^{1,\alpha}$ norm of $g_{ij}$. However, all we know is that the curvature is bounded and moreover, thanks to the derivative estimate of Shi, if we need, we may assume that the derivatives of the curvature tensor are bounded too.

It is well known that in terms of harmonic coordinates, $g_{ij}$ is $C^{1,\alpha}$ for any $\alpha<1$ if the curvature is bounded. In our problem, the metric is changing with time. This makes the discussion tricky.

For simplicity, let's write $\tilde{g}$ for $g(t-\tau/2)$. As for the existence for harmonic coordinates, we have this result of Anderson\cite{Anderson} ,
\begin{thm}
    \label{thm:anderson}
    Given $n\geq 2$ and $\alpha\in (0,1)$, $\Lambda, i_0>0$, one can for each $Q>0$ find $r(n,\alpha,\Lambda,i_0)>0$ such that for any complete Riemannian $n-$manifold $(M,g)$ with
    \begin{eqnarray*}
        \abs{\mbox{Ric}}&\leq& \Lambda\\
        \mbox{inj}&\geq& i_0,
    \end{eqnarray*}
    for any $x\in M$, there exists harmonic coordinates $x_i$ on $B(x,r)$ such that
    \begin{equation*}
        \norm{g_{ij}}_{C^{1,\alpha}}\leq Q.
    \end{equation*}
\end{thm}
We then apply this theorem to $\tilde{g}$. In fact, we get better result
\begin{equation}
    \norm{\tilde{g}_{ij}}_{C^{k,\alpha}}\leq Q(K),
    \label{eqn:better}
\end{equation}
since we have derivatives of Ricci tensor bounded. By choosing a smaller $\tau$ if necessary, we can assume that the harmonic coordinates exist on $B_0(x,\sqrt{\tau/2})$.

To extend our estimate to $s\in [t-\tau/2,t] $, we need the following lemma
\begin{lem}\label{lem:gamma}
    If $\Gamma^i_{jk}$ denotes the Christoffel symbol of $g(s)$ for $s\in [t-\tau/2,t]$, then
    \begin{equation*}
        \abs{\Gamma^i_{ij}}\leq C.
    \end{equation*}
\end{lem}
\begin{proof}
    Write $\tilde{\Gamma}^i_{jk}$ for the Christoffel symbol of $\tilde{g}$. We have
    \begin{equation*}
        \abs{\tilde{\Gamma}^i_{ij}}\leq C.
    \end{equation*}
    Recall that $\pfrac{g_{ij}}{t}=-2h_{ij}$, hence
    \begin{equation}\label{eqn:x2}
        \pfrac{}{t}\Gamma^i_{jk}=-g^{il}(\nabla_j h_{kl}+\nabla_k h_{jl}-\nabla_l h_{jk}).
    \end{equation}
    The lemma follows from the assumption that the right hand side is bounded.
\end{proof}

The following is a basic formula in Riemannian geometry,
\begin{equation}\label{eqn:x1}
    \pfrac{g_{ij}}{x_k}=g_{lj}\Gamma^l_{ki}+g_{il}\Gamma^l_{kj}.
\end{equation}
Due to this formula, we know $g(s)_{ij}$ have bounded first order derivatives in our coordinates. By induction and repeated differentiating equation (\ref{eqn:x1}), the $k-$th derivatives of $g_{ij}$ are bounded if $(k-1)-$th derivatives of $\Gamma^i_{jk}$ are.

To obtain higher order derivatives estimate of $\Gamma^i_{jk}$, we can argue like Lemma \ref{lem:gamma}. We know the derivatives of $\tilde{\Gamma}^i_{jk}$ are bounded. It suffices to control the time derivatives. Take partial derivative of equation (\ref{eqn:x2}). For simplicity, we omit similar terms since we are interested in an upper bound only.
\begin{eqnarray*}
    \pfrac{}{t}(\pfrac{\Gamma^i_{jk}}{x_p})&=& -\pfrac{g^{il}}{x_p}\nabla_j h_{kl}+\pfrac{}{x_p}(\nabla_jh_{kl}) +\cdots \\
    &=& -\frac{\partial g^{il}}{\partial x_p}\nabla_j h_{kl}+\nabla_p\nabla_j h_kl+ \Gamma*\nabla h +\cdots \\
\end{eqnarray*}

The right hand side is bounded. Therefore, we proved the first derivatives of $\Gamma^i_{jk}$, hence the second derivatives of $g(s)_{ij}$ are bounded. Take one more partial derivative of equation (\ref{eqn:x2}) and write $\nabla\nabla\nabla h$ for $\pfrac{}{x_*}\nabla\nabla h$. Since all covariant derivatives of $h$ are bounded, first derivatives of $\Gamma$ and second derivatives of $g_{ij}$ are bounded, it follows that the second derivatives of $\Gamma^i_{ij}$ are bounded. We can repeat the above argument to get bounds for $k-$th derivatives of $g(s)_{ij}$.

Due to the evolution equation, $\pfrac{g_{ij}}{t}=-2h_{ij}$. Since we have bounds on any finite order derivatives of $\Gamma^i_{ij}$, all partial derivatives of $h_{ij}$, hence those of $\pfrac{g_{ij}}{t}$ are bounded. These are strong enough so that we can apply the routine estimate of parabolic equation.

\end{document}